\documentclass[12pt]{amsart}%
\usepackage{amsfonts}
\usepackage{mathtools}
\usepackage{amsmath}
\usepackage{amssymb}
\usepackage{amsthm}
\usepackage{newlfont}
\usepackage{graphicx}
\usepackage{amssymb,amsmath,xcolor,graphicx,xspace,colortbl, rotating}
\usepackage{enumitem}
\usepackage{textcomp}%
\providecommand{\U}[1]{\protect\rule{.1in}{.1in}}
\graphicspath{{proof_latest teX_graphics/}{proof_latest teX_tcache/}{proof_latest teX_gcache/}}
\DeclareGraphicsExtensions{.pdf,.eps,.ps,.png,.jpg,.jpeg}
\newtheorem {theorem}{Theorem}[section]

\newtheorem {claim}[theorem]{Claim}

\newtheorem {condition}[theorem]{Condition}

\newtheorem {corollary}[theorem]{Corollary}

\newtheorem {definition}[theorem]{Definition}

\newtheorem {lemma}[theorem]{Lemma}

\newtheorem {proposition}[theorem]{Proposition}

\numberwithin{equation}{section}
\begin{document}
\title[Almost Linear Equations]{$C^{2,\alpha}$ estimates for solutions to almost Linear elliptic equations}
\author{Arunima Bhattacharya AND Micah Warren }

\begin{abstract}
In this paper, we show explicit $C^{2,\alpha}$ interior estimates for
viscosity solutions of fully non-linear, uniformly elliptic equations, which
are close to linear equations and we also compute an explicit bound for the closeness.

\end{abstract}
\maketitle

\section{Introduction}

In this paper, we derive an a priori interior $C^{2,\alpha}$ estimate for
viscosity solutions of the non-linear, uniformly elliptic equation%
\begin{equation}
F(D^{2}u)=f(x), \label{F}%
\end{equation}
under the assumption that $f(x)\in C^{\alpha}$ and $F$ is $C^{1}$-close to a
linear operator.

For viscosity solutions of second order, fully non-linear equations of the
form
\begin{equation}
F(D^{2}u)=0 \label{f}%
\end{equation}
where $F$ is concave and uniform elliptic, the landmark estimate is that of
Krylov and Evans, who proved $C^{2,\alpha}$ estimates from $C^{2}$ estimates
\cite{K}, \cite{E}. For general $F$, the fundamental results on the regularity
of solutions to fully non-linear uniformly elliptic equations of the form
(\ref{f}) include interior $C^{\alpha}$ estimates of \cite{KS} and interior
$C^{1,\alpha}$ estimate of \cite{CC}. The structure of $F$ plays a key role in
deriving higher order estimates for fully non-linear elliptic equations of the
forms (\ref{F}) and (\ref{f}). In \cite{NV}, the authors produced
counterexamples to Evans-Krylov type estimates for general fully non-linear
equations. In fact, solutions need not even be $C^{1,1},$ \cite{NVSL}.

Prior to Krylov and Evans, few fully non-linear equations where known to enjoy
a $C^{1,1}$ to $C^{2,\alpha}$ regularity boost. The Monge-Amp\`{e}re equation
was shown to have this property (even stronger, $C^{1,1}$ to $C^{3})$
following Calabi's calculation \cite{C}. Other results, requiring stronger
conditions on $D^{2}F,$ are mentioned in \cite[pg 335.]{E}. If the linearized
operator for $F$ satisfies a Cordes-Nirenberg condition, one can also obtain
this boosting (see Section \ref{cnsection}). Since the 1980s, it has been a
challenge to find equations with the regularity boosting property that are
niether convex nor concave, see for example \cite{CafYuan}, \cite{Yuan2001},
\cite{CC03} , \cite{Collins}, \cite{StreetsWarren}, \cite{Pingali}. Savin
\cite{SV} proved interior $C^{2,\alpha}$ (and higher) estimates for viscosity
solutions of (\ref{f}) that are sufficiently close to a quadratic polynomial,
for $F$ smooth. When full regularity is not available, partial regularity
results can be found, see \cite{SSA}.

Here, we consider a space of uniformly elliptic, non-linear equations of the
forms (\ref{f}) and (\ref{F}) where we assume that $F$ is uniformly
differentiable and $DF$ lies in a set of diameter $\varepsilon_{0}$. We
formally define this property of $F$ in definition \ref{definition 1.1}. We
show that given ellipticity constants and an $\alpha$ $\in(0,1)$ of your
choice, there is a universal constant $\varepsilon_{0}\left(  n,\lambda
,\Lambda,\alpha\right)  $ guaranteeing $C^{2,\alpha}$ regularity.

\bigskip Differentiating (\ref{f}) with respect to a direction $i$, one sees
that $u_{i}$ solves a linear equation with bounded measurable coefficients
(now depending on $x,$ not $D^{2}u)$. One then hopes to achieve $C^{1,\alpha}$
estimates on $u_{i},$ yielding $C^{2,\alpha}$ estimates on $u.$ In particular,
it may be possible to apply estimates of Cordes and Nirenberg from the 1950s:
Any solution $v$ of a linear equation
\begin{equation}
a^{ij}(x)v_{ij}(x)=0 \label{basiclinear}%
\end{equation}
with coefficients close to $\delta^{ij}$ will enjoy $C^{1,\alpha}$ regularity.
Thus when a solution is already $C^{3},$ universal interior $C^{2,\alpha}$
estimates should follow by the Cordes-Nirenberg theory. A delicate analysis of
the Dirichlet boundary value problem, approximating $u$ with $C^{3}$
mollifications on the boundary should also yield the estimates when $u$ is not
known to be $C^{3},$ cf.\cite[Section 7]{E}. The closeness constants of
Cordes-Nirenberg are explicit and mildly restrictive, in fact much less
restrictive than ours. As the historical literature is not widely known, we
discuss the Cordes-Nirenberg results in more detail in Section \ref{cnsection}.

Note that our result is stated for every $\alpha\in(0,1)$. Also, note that for
equation (\ref{F}) one cannot hope to differentiate either side of (\ref{F})
if the right hand side is merely $C^{\alpha},$ so the regularity theory cannot
be immediately reduced to the Cordes-Nirenberg theory. Our methods for proving
\ref{Theorem 1.2} are much different in nature than the proof of Cordes and
Nirenberg: we use the method of constructing approximating polynomials,
instead of integral estimates. In Theorem \ref{Theorem 1.3}, we prove interior
$C^{2,\alpha}$ estimates for solutions of (\ref{F}) using our $C^{2,\alpha}$
estimates for (\ref{f}) together with estimates found in \cite{CC}.

This paper is divided into the following sections. In the remainder of this
section we state definitions and our main results. In section 2, we prove
Theorem \ref{Theorem 1.2} and in section 3, we prove Theorem \ref{Theorem 1.3}%
. In section 4 we explicity state and prove an often used result involving
H\"{o}lder estimates and in section 5 we further discuss the Cordes-Nirenberg
regularity and some applications of Cordes-Nirenberg regularity to equations
of the form (\ref{f}).

\subsection{Definitions and notations}

We first define a few terms that we will be using to state the properties of
the operator $F$.

\begin{condition}
Throughout this paper we make the assumption
\begin{equation}
F(\mathbf{0})=0. \label{condi1}%
\end{equation}

\end{condition}

\begin{definition}
\label{definition 1.1} We define the uniformly elliptic, non-linear operator
$F$ to be \textbf{almost linear with constant $\varepsilon$} if
\begin{equation}
\left\Vert DF(M)-DF(N)\right\Vert \leq\varepsilon\label{Cm1}%
\end{equation}
for all $M,N\in S_{n}$ where $S_{n}$ is the space of all real symmetric
$n\times n$ matrices. We define $\varepsilon$ to be the \textbf{closeness
constant} of $F$.

We say that $F$ is $\lambda,\Lambda$ elliptic if
\[
F(M)+\lambda\left\Vert P\right\Vert \leq F(M+P)\leq F(M)+\Lambda\left\Vert
P\right\Vert
\]
for all positive matrices $P.$ To be clear, for matrices and their dual
(linear operators) we use $\left\Vert {}\right\Vert $ to denote the $\left(
L^{2},L^{2}\right)  $ norm, that is
\[
\left\Vert M\right\Vert =\sup_{x\leq1}\left\Vert Mx\right\Vert .
\]

\end{definition}

\begin{theorem}
\label{Theorem 1.2} Given $\lambda$, $\Lambda$, and $0<\bar{\alpha}<1$ there
exists a universal constant $\varepsilon_{0}(n,\lambda,\Lambda,\bar{\alpha
})>0$ such that if $F$ is almost linear with constant $\varepsilon_{0}$ and
$u\in C(B_{1})$ is a viscosity solution of (\ref{f}) on $B_{1}$, then $u\in
C^{2,\bar{\alpha}}(B_{\frac{1}{4\Lambda}})$ and satisfies the following
estimate
\begin{equation}
||D^{2}u||_{C^{\bar{\alpha}}(B_{\frac{1}{4\Lambda}})}\leq C_{1}%
||u||_{L^{\infty}(B_{1})} \label{EE}%
\end{equation}
where
\begin{equation}
C_{1}=\left(  1+n+4n^{2}+\frac{1}{\lambda}\varepsilon_{0}\frac{25}{8}%
n^{2}\right)  (1+\frac{3}{1-r_{0}^{\bar{\alpha}}})\frac{2^{\bar{\alpha}}%
}{r_{0}^{1+\bar{\alpha}}}\Lambda^{2+\bar{\alpha}}\left(  2+2^{2+\bar{\alpha}%
}\right)  ^{2}. \label{con1}%
\end{equation}
The constant $\varepsilon_{0}$ is determined in (\ref{eps_0}), (\ref{Vep})
\end{theorem}

\begin{theorem}
\label{Theorem 1.3} Given $\lambda$, $\Lambda$, and $0<\alpha<\bar{\alpha}<1$,
suppose that $F$ is almost linear with constant $\varepsilon_{0}$ for the same
constant $\varepsilon_{0}(n,\lambda,\Lambda,\bar{\alpha})$ as in Theorem
\ref{Theorem 1.2} and $u\in C(B_{1})$ is viscosity solution of (\ref{F}) on
$B_{1}$. If $f\in C^{\alpha}(B_{1})$, then $u\in C^{2,\alpha}(B_{1/2})$ and
the following estimate holds
\begin{equation}
||u||_{C^{2,\alpha}(B_{1/2})}\leq C_{2}(\left\Vert u\right\Vert _{L^{\infty
}(B_{1})}+\left\Vert f\right\Vert _{C^{\alpha}(B_{1})}) \label{AAA}%
\end{equation}
where $C_{2}$ depends on $n,\lambda,\Lambda,C_{1}$, $\alpha,\bar{\alpha}$.
\end{theorem}


The methods involved in our proof include comparing equation (\ref{f}) to the
Laplace equation with boundary data equal to a mollification of $u$. We use
the Krylov-Safanov Theorem \cite{KS} along with harmonic estimates to
construct a quadratic polynomial that separates from $u$ to order
$r^{2+\alpha}$ on the ball of radius $r$. This is used in the construction of
an iterative sequence of quadratic polynomials that leads to our desired
estimate in the first theorem. The proof of Theorem \ref{Theorem 1.3} uses
arguments from $W^{2,p}$ regularity found in \cite[Chapter 7]{CC} .

\section{Proof of Theorem \ref{Theorem 1.2}}

By calculus,
\begin{align*}
F(N)-F(\mathbf{0)}  &  =\int_{0}^{1}DF(tN)\cdot Ndt\\
&  =\int_{0}^{1}\left(  DF(tN)-DF(\mathbf{0})\right)  \cdot Ndt+DF(\mathbf{0}%
)\cdot N.
\end{align*}
Thus, from (\ref{condi1}) and (\ref{Cm1})%
\begin{equation}
\left\vert F(N)-DF(\mathbf{0})\cdot N\right\vert \leq\varepsilon_{0}\left\Vert
N\right\Vert . \label{Cm22}%
\end{equation}
With this is mind we begin with the following Lemma.

\begin{lemma}
\label{lemma1}Given $\bar{\alpha},\lambda,\Lambda$ there exist universal
constants $\tilde{\varepsilon}_{0}(n,\lambda,\Lambda,\bar{\alpha})>0$ and
$r_{0}(n,\bar{\alpha})>0$, such that if the $\lambda,\Lambda$ elliptic
operator $F$ satisfies
\begin{equation}
|F(N)-tr(N)|<\tilde{\varepsilon}_{0}\left\Vert N\right\Vert \label{CM0}%
\end{equation}
for all $N\in S_{n}$, then for any viscosity solution $u\in C(B_{1})$ of
(\ref{f}) in $B_{1}(0)$, we can find a polynomial $P$ of degree 2 satisfying
\begin{align}
F(D^{2}P)  &  =0\nonumber\\
\sup_{B_{r_{0}}}|u-P|  &  \leq r_{0}^{2+\bar{\alpha}}||u||_{L^{\infty}(B_{1}%
)}\nonumber\\
||P||_{L^{\infty}(B_{1})}  &  \leq C_{0}||u||_{L^{\infty}(B_{1})}. \label{lem}%
\end{align}
We compute the explicit values of the universal constants to be

\begin{enumerate}
[label=(\roman*)]{}{}

\item $r_{0}=\left(  \frac{3}{250n^{3}}\right)  ^{\frac{1}{1-\bar{\alpha}}}$

\item $C_{0}=1+n+4n^{2}+\frac{1}{\lambda}\tilde{\varepsilon}_{0}\frac{25}%
{8}n^{5/2}$

\item $\tilde{\varepsilon}_{0}=\min\left\{  \lambda\frac{2}{25n^{2}}%
r_{0}^{\bar{\alpha}},\left(  \frac{1}{2}\right)  ^{1+6/\alpha_{0}}%
\frac{\lambda}{K_{2}}\frac{1}{K_{1}^{3/\alpha_{0}}}r_{0}^{\left(
2+\bar{\alpha}\right)  (1+3/\alpha_{0})}\right\}  $ where $K_{1}$, $\alpha
_{0}$, $K_{2}$ are defined in (\ref{KS}), (\ref{fwalp}),and (\ref{K2p}) respectively.
\end{enumerate}
\end{lemma}

The required constant $\alpha_{0}$ is defined in the proof of the Lemma, and
will require the Krylov-Safanov Theorem, so we state that here.

\begin{theorem}
\cite[Theorem 1]{KS} [Krylov-Safanov] Let $u\in C^{0}$ be a viscosity solution
of $S(\frac{\lambda}{n},\Lambda,0)=0$ in $B_{1}$. Then $u$ is H\"{o}lder
continuous and
\begin{equation}
||u||_{C^{\alpha_{0}}(B_{1/2})}\leq C(\frac{\lambda}{n},\Lambda
)||u||_{L^{\infty}(B_{1})} \label{fwalp}%
\end{equation}
with (small) $\alpha_{0}=\alpha_{0}(\frac{\lambda}{n},\Lambda)>0$.
\end{theorem}

We will apply the following result to the Laplace operator to determine the
constant $K_{2}.$ We state a weaker version than in \cite[Theorem 9.5]{CC}.

\begin{theorem}
\cite[Theorem 9.5]{CC} Let $g$ be a smooth function in $\bar{B}_{1}.$ If $u\in
C^{3}(\bar{B}_{1})$ is a solution of
\[%
\begin{cases}
\Delta u=0\hspace{0.2cm}\text{in }\bar{B}_{1}\\
u=g\text{ \ on }\partial B_{1}%
\end{cases}
\]
then
\begin{equation}
\left\Vert u\right\Vert _{C^{2}(\bar{B}_{1})}\leq C^{\prime}\left\Vert
g\right\Vert _{C^{3}(\partial B_{1}).} \label{boundaryestimate}%
\end{equation}
where $C^{\prime}$ is a universal constant.
\end{theorem}

\begin{proof}
[Proof of Lemma 2.1]Let's denote $||u||_{L^{\infty}(B_{1})}=M$. We consider a
function $h$ that satisfies the following boundary value problem:
\begin{equation}%
\begin{cases}
\Delta h=0\hspace{0.2cm}\text{in }\bar{B}_{4/5}\\
h=u^{\gamma}\text{ \ on }\partial B_{4/5}%
\end{cases}
.\nonumber
\end{equation}
Here $u^{\gamma}$ refers to a mollification of $u$ for some $\gamma\in
(0,1/5)$, defined by
\[
u^{\gamma}=\eta_{\gamma}\ast u
\]
where
\[
\eta_{\gamma}(x)=\frac{1}{\gamma^{n}}\eta(\frac{x}{\gamma})
\]
and $\eta\in C^{\infty}(\mathbb{R}^{n})$ is given by
\[
\eta(x)=%
\begin{cases}
C\exp\left(  \frac{1}{|x|^{2}-1}\right)  \hspace{0.2cm}if\hspace{0.2cm}|x|<1\\
0\hspace{2.2cm}if\hspace{0.2cm}|x|\geq1
\end{cases}
\]
with the constant $C>0$ being chosen such that $\int_{\mathbb{R}^{n}}\eta
dx=1$. Note that since $u$ is defined on all of $B_{1}$, the mollifier
sequence $u^{\gamma}$ is well defined on $B_{4/5}$ when $\gamma<1/5$ and that
\begin{equation}
||u^{\gamma}||_{L^{\infty}(B_{4/5})}\leq M. \label{basicb}%
\end{equation}
From the Krylov-Safanov theorem, we get the following estimate
\begin{equation}
||u||_{C^{\alpha_{0}}(B_{4/5})}\leq K_{1}M. \label{KS}%
\end{equation}

This implies that $u^{\gamma}$ converges to $u$ uniformly on $\bar{B_{4/5}}$
as $\gamma\rightarrow0$ and satisfies the following estimate:
\begin{equation}
||u^{\gamma}-u||_{L^{\infty}(B_{4/5})}\leq K_{1}\gamma^{\alpha_{0}}M.
\label{KS1}%
\end{equation}

Since $h$ is harmonic and thus analytic there exists a polynomial $P_{0}(x)$
of degree two
\[
P_{0}(x)=h(0)+x\cdot Dh(0)+x\cdot D^{2}h(0)x\label{PH}%
\]
such that for all $|x|<1/2$,
\[
|h(x)-P_{0}(x)|\leq|R_{3}(x)|
\]
where $R_{3}$ is the remainder term of order 3 in the Taylor series expansion
of $h$. Estimates for harmonic functions (cf. \cite[(2.31)]{GT}), considering
(\ref{basicb}) are of the form
\begin{align*}
\sup_{x\in B_{1/5}}\left\vert h_{ijk}(x)\right\vert  &  \leq\frac{n}{1/5}%
\sup_{x\in B_{2/5}}\left\vert h_{ij}(x)\right\vert \\
&  \leq5n\frac{n}{1/5}\sup_{x\in B_{3/5}}\left\vert h_{i}(x)\right\vert \\
&  \leq25n^{2}\frac{n}{1/5}\sup_{x\in B_{4/5}}\left\vert h(x)\right\vert \\
&  \leq125n^{3}M.
\end{align*}
Thus we have on $B_{1/5}$
\[
|h(x)-P_{0}(x)|\leq\frac{125}{3!}n^{3}M\left\vert x\right\vert ^{3}.\text{ }%
\]

Choosing%
\begin{equation}
r_{0}=\left(  \frac{3}{250n^{3}}\right)  ^{\frac{1}{1-\bar{\alpha}}}<<\frac
{1}{5}, \label{R}%
\end{equation}
we have
\begin{equation}
\sup_{B_{r_{0}}}|h(x)-P_{0}(x)|\leq\frac{1}{4}Mr_{0}^{2+\bar{\alpha}}.
\label{BBC}%
\end{equation}

Now from (\ref{CM0}) and $\Delta P_{0}=0$, we see that
\[
\left\vert F(D^{2}P_{0})\right\vert \leq\tilde{\varepsilon}_{0}\left\Vert
D^{2}P_{0}\right\Vert .
\]
So using $\lambda$-ellipticity, there is a $c\in\lbrack-\tilde{\varepsilon
}_{0},\tilde{\varepsilon}_{0}]$ such that the quadratic polynomial%
\begin{equation}
P(x)=P_{0}(x)+\frac{|x|^{2}}{2\lambda}c\left\Vert D^{2}h(0)\right\Vert
\label{difp}%
\end{equation}
satisfies
\[
F(D^{2}P)=0.\label{la}%
\]
Using harmonic estimates again we see that
\begin{equation}
\left\Vert D^{2}h(0)\right\Vert \leq\frac{25}{4}n^{2}M. \label{he}%
\end{equation}
Bringing in (\ref{difp}) we see
\begin{equation}
\sup_{B_{r_{0}}}|h-P|<\sup_{B_{r_{0}}}|h-P_{0}|+\frac{r_{0}{}^{2}}{2\lambda
}\tilde{\varepsilon}_{0}\frac{25}{4}n^{2}M. \label{ABC}%
\end{equation}
Insisting on a choice of $\tilde{\varepsilon}_{0}$ such that
\begin{equation}
\tilde{\varepsilon}_{0}\leq\frac{\lambda}{2}\frac{r_{0}^{\bar{\alpha}}}%
{\frac{25}{4}n^{2}}=\lambda\frac{2}{25n^{2}}\left(  \frac{3}{250n^{3}}\right)
^{\frac{\bar{\alpha}}{1-\bar{\alpha}}} \label{VA}%
\end{equation}
we conclude from (\ref{ABC}) and (\ref{BBC})
\begin{equation}
\sup_{B_{r_{0}}}|h-P|\leq\frac{1}{2}Mr_{0}^{2+\bar{\alpha}}. \label{abcd}%
\end{equation}
Again using harmonic estimates (\ref{he}), we get the following estimate for
$P$:
\begin{align}
||P||_{L^{\infty}(B_{1})}  &  \leq C_{0}M,\nonumber\\
C_{0}  &  =1+n+\frac{25}{4}n^{2}+\frac{1}{2\lambda}\tilde{\varepsilon}%
_{0}\frac{25}{4}n^{2}\label{C}\\
&  =1+n+\frac{25}{4}\left(  1+\frac{1}{2\lambda}\tilde{\varepsilon_{0}%
}\right)  n^{2}.
\end{align}

Next, by (\ref{CM0}) for $x\in B_{4/5}$ we have
\begin{align}
|F(D^{2}h(x))|  &  =|F(D^{2}h)-\Delta h+\Delta h)|\nonumber\\
&  =|F(D^{2}h)-Tr(D^{2}h)|\\
&  \leq\tilde{\varepsilon}_{0}||D^{2}h||_{L^{\infty}\bar{(B_{4/5})}}.
\label{NN}%
\end{align}

Now recall (\ref{boundaryestimate}):
\[
||D^{2}h||_{L^{\infty}\bar{(B_{4/5})}}\leq C^{\prime}||u_{\gamma}||_{C^{3}%
\bar{(B_{4/5})}}.
\]
We compute the value of $\left\Vert u_{\gamma}\right\Vert _{C^{3}\bar
{(B_{4/5})}}$.

Let $p$ be a multi-index such that $|p|=3$. For any $x\in\bar{B_{4/5}}$ we
observe the following:
\begin{align*}
|D^{p}(u_{\gamma}(x))|  &  =\left\vert D^{p}(\eta_{\gamma})\ast
u(x)\right\vert =\left\vert \int_{B_{1}}D^{p}\eta_{\gamma}%
(x-y)u(y)dy\right\vert \\
&  \leq\sup_{y\in B_{1}}|u(y)|\int_{B_{1}}|D^{p}\eta_{\gamma}(x-y)|dy\\
&  \leq M\int_{B_{1}}\left\vert \frac{1}{\gamma^{n+3}}D^{p}\eta(\frac
{x-y}{\gamma})\right\vert dy.
\end{align*}
We do a change of variable $z=\frac{x-y}{\gamma}$ to reduce the above
expression to
\[
\leq M\frac{1}{\gamma^{3}}\int_{B_{1}}|\frac{1}{\gamma^{n}}D^{p}\eta
(z)\gamma^{n}|dz=M\frac{1}{\gamma^{3}}\int_{B_{1}}|D^{p}\eta(z)|dz.
\]

This shows that
\begin{equation}
||D^{2}h||_{L^{\infty}\bar{(B_{4/5})}}\leq C^{\prime}M\frac{1}{\gamma^{3}}%
\sup_{\left\vert p\right\vert =3}\int_{\mathbb{R}^{n}}|D^{p}\eta(z)|dz.
\label{K}%
\end{equation}
Let's define
\begin{equation}
K_{2}=C^{\prime}\sup_{\left\vert p\right\vert =3}\int_{\mathbb{R}^{n}}%
|D^{p}\eta(z)|dz \label{K2p}%
\end{equation}
so that
\begin{equation}
||D^{2}h||_{L^{\infty}\bar{(B_{4/5})}}\leq K_{2}M\frac{1}{\gamma^{3}}.
\label{K2}%
\end{equation}

Using uniform ellipticity, (\ref{NN}), and (\ref{K2}) we see that the
following inequalities hold on $B_{4/5}$:
\begin{align*}
F(D^{2}h+D^{2}(\frac{\tilde{\varepsilon}_{0}}{2\lambda}K_{2}M\frac{1}%
{\gamma^{3}}(1-|x|^{2}))  &  \leq0.\\
F(D^{2}h-D^{2}(\frac{\tilde{\varepsilon}_{0}}{2\lambda}K_{2}M\frac{1}%
{\gamma^{3}}(1-|x|^{2}))  &  \geq0.
\end{align*}
Using comparison principles \cite[Theorem 17.1]{GT} and (\ref{KS1}) we see
that for all $x\in B_{4/5}$ we have:
\begin{equation}
|u(x)-h(x)|\leq K_{1}M\gamma^{\alpha_{0}}+\frac{\tilde{\varepsilon}_{0}%
}{2\lambda}K_{2}M\frac{1}{\gamma^{3}}. \label{BB}%
\end{equation}

On combining (\ref{BB}), (\ref{abcd}) we see that%
\begin{align}
\sup_{B_{r_{0}}}|u-P|  &  <\sup_{B_{r_{0}}}|u-h|+\sup_{B_{r_{0}}%
}|h-P|\nonumber\\
&  <K_{1}M\gamma^{\alpha_{0}}+\frac{\tilde{\varepsilon}_{0}}{2\lambda}%
K_{2}M\frac{1}{\gamma^{3}}+\frac{1}{2}Mr_{0}^{2+\bar{\alpha}}. \label{NW}%
\end{align}

The right hand side of (\ref{NW}) will be no greater than $Mr_{0}%
^{2+\bar{\alpha}}$ provided
\[
K_{1}\gamma^{\alpha_{0}}+\frac{\tilde{\varepsilon}_{0}}{2\lambda}K_{2}\frac
{1}{\gamma^{3}}\leq\frac{1}{2}r_{0}^{2+\bar{\alpha}}%
\]
for some choice of $\gamma$ and $\tilde{\varepsilon}_{0}$. While this could be
optimized with some messy calculus, we scare up constants as follows. Choose
\begin{equation}
\gamma=\left(  \frac{\frac{1}{4}r_{0}^{2+\bar{\alpha}}}{K_{1}}\right)
^{1/\alpha_{0}}\label{delta}%
\end{equation}
so that
\[
K_{1}\gamma^{\alpha_{0}}=\frac{1}{4}r_{0}^{2+\bar{\alpha}}%
\]
and then we want
\[
\frac{\tilde{\varepsilon}_{0}}{2\lambda}K_{2}\frac{1}{\left(  \frac{\frac
{1}{4}r_{0}^{2+\bar{\alpha}}}{K_{1}}\right)  ^{3/\alpha_{0}}}\leq\frac{1}%
{4}r_{0}^{2+\bar{\alpha}}%
\]
so we choose%
\begin{align}
\tilde{\varepsilon}_{0} &  \leq\frac{1}{4}r_{0}^{2+\bar{\alpha}}\frac
{2\lambda}{K_{2}}\left(  \frac{\frac{1}{4}r_{0}^{2+\bar{\alpha}}}{K_{1}%
}\right)  ^{3/\alpha_{0}}\nonumber\\
&  =\left(  \frac{1}{2}\right)  ^{1+6/\alpha_{0}}r_{0}^{\left(  2+\bar{\alpha
}\right)  (1+3/\alpha_{0})}\frac{\lambda}{K_{2}}\frac{1}{K_{1}^{3/\alpha_{0}}%
}\label{VE}%
\end{align}
where $K_{1}$, $\alpha_{0}$ and $K_{2}$ are defined in (\ref{KS}) and
(\ref{K2}) respectively and $r_{0}$ defined by (\ref{R}), From (\ref{VA}) and
(\ref{VE}) we see that
\begin{equation}
\tilde{\varepsilon}_{0}=\min\left\{
\begin{array}
[c]{c}%
\lambda\frac{2}{25n^{2}}\left(  \frac{3}{250n^{3}}\right)  ^{\frac{\bar
{\alpha}}{1-\bar{\alpha}}},\\
\left(  \frac{1}{2}\right)  ^{1+6/\alpha_{0}}\frac{\lambda}{K_{2}}\frac
{1}{K_{1}^{3/\alpha_{0}}}\left(  \frac{3}{250n^{3}}\right)  ^{\frac{\left(
2+\bar{\alpha}\right)  (1+3/\alpha_{0})}{1-\bar{\alpha}}}%
\end{array}
\right\}  .\label{Vep}%
\end{equation}

\end{proof}

We now make a proposition similar to the statement of Theorem 1.2, but with
the operator close to the Laplacian. Throughout this proof the constants
$C_{0}$ and $r_{0}$ will refer to the constants obtained in (\ref{C}) and
(\ref{R}) respectively.

\begin{proposition}
\label{prop}Given $\bar{\alpha},\lambda,\Lambda,$ if the $\lambda,\Lambda$
elliptic operator $F$ satisfies
\begin{equation}
|F(N)-tr(N)|<\tilde{\varepsilon}_{0}\left\Vert N\right\Vert
\end{equation}
for all $N\in S_{n}$, then any viscosity solution $u\in C(B_{1})$ of (\ref{f})
will be in $C^{2,\bar{\alpha}}(B_{1/4})$ and satisfy the following estimate
\[
||u||_{C^{2,\bar{\alpha}}(B_{1/4})}\leq\tilde{C}_{1}||u||_{L^{\infty}(B_{1})}%
\]
for
\begin{equation}
\tilde{C}_{1}=C_{0}(1+\frac{3}{1-r_{0}^{\bar{\alpha}}})\frac{2^{\bar{\alpha}}%
}{r_{0}^{1+\bar{\alpha}}}\left(  2+2^{2+\bar{\alpha}}\right)  ^{2} \label{eee}%
\end{equation}
where $C_{0},r_{0}$ and $\tilde{\varepsilon}_{0}$ are as stated in Lemma
\ref{lemma1}.
\end{proposition}

\begin{proof}
We first prove that the $C^{2,\bar{\alpha}}$ estimate holds at the origin. As
before, we denote $||u||_{L^{\infty}(B_{1})}=M$. \newline We prove that there
exists a polynomial $P$ of degree 2 such that
\begin{align}
|u(x)-P(x)|  &  \leq MC_{0}^{\prime}|x|^{2+\bar{\alpha}}\hspace{0.2cm}\forall
x\in B_{1}\label{e1}\\
F(D^{2}P)  &  =0\nonumber\\
||P||_{L^{\infty}(B_{1})}  &  \leq MC_{0}^{\prime}\nonumber
\end{align}
where $C_{0}^{\prime}=C_{0}(1+\frac{3}{1-r_{0}^{\bar{\alpha}}})\frac{1}%
{r_{0}^{1+\bar{\alpha}}}$. In order to prove the existence of such a
polynomial $P$, we need the following claim.

\begin{claim}
\label{claim_CL2} There exists a sequence of polynomials $\{P_{k}%
\}_{k=1}^{\infty}$ of degree 2 such that
\begin{align}
F(D^{2}P_{k})  &  =0\label{e0}\\
||u-P_{k}||_{L^{\infty}(B_{r_{0}^{k}})}  &  \leq Mr_{0}^{k(2+\bar{\alpha})}
\label{e2}%
\end{align}
where $F$ and $u$ are as defined in Proposition \ref{prop}.
\end{claim}

We first prove the claim.

\begin{proof}
: Let $P_{0} =0$. Then (\ref{e2}) holds good for the $k =0$ case. We assume
that (\ref{e2}) holds for $k \leq i$ and we prove it for $k =i +1.$

Consider
\[
v_{i}(x)=\frac{u(r_{0}^{i}x)-P_{i}(r_{0}^{i}x)}{r_{0}^{2i}}%
\]
for all $x\in B_{1}$. Define
\[
F_{i}(N)=F(N+D^{2}P_{i})
\]
for all $N\in S_{n}$. Since $F(D^{2}P_{i})=0$ we see that $F_{i}(D^{2}%
v_{i})=0$. Since
\[
||u-P_{i}||_{L^{\infty}(B_{r_{0}^{i}})}\leq Mr_{0}^{i(2+\bar{\alpha})},
\]
we observe that
\[
||v_{i}||_{L^{\infty}(B_{1})}\leq\frac{Mr_{0}^{i(2+\bar{\alpha})}}{r_{0}^{2i}%
}=Mr_{0}^{i\bar{\alpha}}.
\]
Note that the operator $F_{i}$ satisfies the same properties as the operator
$F$:
\[
|DF_{i}(M)-DF_{i}(N)|=|DF(M+D^{2}P_{i})-DF(N+D^{2}P_{i})|\leq\tilde
{\varepsilon}_{0}%
\]
and $F_{i}$ also has the same ellipticity constants as $F$. We apply Lemma
\ref{lemma1} to the equation $F_{i}(D^{2}v_{i})=0$. This gives us the
existence of a quadratic polynomial
\begin{equation}
\bar{P_{i}}=a_{i}+\vec{b}_{i}\cdot x+x^{T}\mathbf{c}_{i}\cdot x \label{pbar}%
\end{equation}
such that
\begin{align}
F_{i}(D^{2}\bar{P_{i}})  &  =0\label{est}\\
||v_{i}-\bar{P_{i}}||_{L^{\infty}(B_{r_{0}})}  &  \leq Mr_{0}^{i\bar{\alpha}%
}r_{0}^{(2+\bar{\alpha})}\label{est2}\\
|\bar{|P}_{i}||_{L^{\infty}(B_{1})}  &  \leq C_{0}Mr_{0}^{i\bar{\alpha}}.
\label{est3}%
\end{align}
We conclude immediately from (\ref{est3}) that
\begin{align}
\left\vert a_{i}\right\vert  &  \leq C_{0}Mr_{0}^{i\bar{\alpha}}\label{est4}\\
\left\Vert b_{i}\right\Vert  &  \leq C_{0}Mr_{0}^{i\bar{\alpha}}\nonumber\\
\left\Vert c_{i}\right\Vert  &  \leq C_{0}Mr_{0}^{i\bar{\alpha}}.\nonumber
\end{align}

Next, we define
\begin{equation}
P_{i +1} =P_{i} +r_{0} ^{2i}\bar{P}_{i}(r_{0} ^{ -i}x). \label{e5}%
\end{equation}

From (\ref{est}) we see that
\[
F(D^{2}P_{i+1})=F_{i}(D^{2}\bar{P}_{i})=0
\]
and on substituting the expression for $v_{i}$ into (\ref{est2}) we see that
\[
||\frac{u(r_{0}^{i}x)-P_{i}(r_{0}^{i}x)}{r_{0}^{2i}}-\bar{P}_{i}||_{L^{\infty
}(B_{r_{0}})}\leq Mr_{0}^{i\bar{\alpha}}r_{0}^{(2+\bar{\alpha})}%
\]
which reduces to
\[
||u-P_{i+1}||_{L^{\infty}(B_{r_{0}^{i+1}})}\leq Mr_{0}^{(i+1)(2+\bar{\alpha}%
)}.
\]
This completes the inductive construction of the quadratic polynomial
sequence. Hence the claim \ref{claim_CL2}.
\end{proof}

Using the above claim, we return to proving Proposition \ref{prop}.\newline We
show that this sequence $\{P_{k}\}_{k=1}^{\infty}$ is convergent and
$\lim_{k\rightarrow\infty}P_{k}=P$ is the required polynomial in (\ref{e1}).

From (\ref{e5}), (\ref{pbar}) we see that
\begin{equation}
P_{i+1}-P_{i}=r_{0}^{2i}a_{i}+r_{0}^{i}\vec{b}_{i}\cdot x+x^{T}\mathbf{c}%
_{i}\cdot x. \label{Pmarker}%
\end{equation}
Inequality (\ref{est3}) guarantees that the series $\sum_{i=1}^{\infty
}(P_{i+1}-P_{i})$ is bounded by a convergent geometric series
\[
|P_{i+1}-P_{i}|\leq MC_{0}r_{0}^{i\bar{\alpha}}.
\]
Hence the telescopic series $\sum_{i=1}^{\infty}(P_{i+1}-P_{i})$ converges
uniformly on the unit ball and we define
\[
P=\lim_{i\rightarrow\infty}P_{i}=\sum_{i=1}^{\infty}(P_{i+1}-P_{i}).
\]
Note that $F(D^{2}P)=0$ as $F(D^{2}P_{i})=0$ for all $i$. The limit $P$ will
be a quadratic polynomial as well.

For $x\in B_{r_{0}^{i}}$ we have, using (\ref{Pmarker}), (\ref{est4}) \
\begin{align*}
|P(x)-P_{i}(x)|  &  \leq\\
\sum_{j=i}^{\infty}|P_{j+1}-P_{j}|  &  \leq C_{0}M\sum_{j=i}^{\infty}%
(r_{0}^{2j}r_{0}^{j\bar{\alpha}}+r_{0}^{j}r_{0}^{j\bar{\alpha}}r_{0}^{i}%
+r_{0}^{i}r_{0}^{j\bar{\alpha}}r_{0}^{i})\\
&  =C_{0}M\left(  \frac{r_{0}^{\left(  2+\bar{\alpha}\right)  i}}%
{1-r_{0}^{2+\bar{\alpha} }}+\frac{r_{0}^{\left(  1+\bar{\alpha}\right)  i}%
}{1-r_{0}^{1+\bar{\alpha}}}r_{0}^{i}+\frac{r_{0}^{i\bar{\alpha}}}%
{1-r_{0}^{\bar{\alpha}}}r_{0}^{2i}\right) \\
&  \leq3C_{0}M\frac{1}{1-r_{0}^{\bar{\alpha}}}r_{0}^{\left(  2+\bar{\alpha
}\right)  i}.
\end{align*}

If we fix $x\in B_{1}$, we can choose an integer $i$ such that
\[
r_{0}^{i+1}<\left\Vert x\right\Vert \leq r_{0}^{i}.
\]
Then we have the estimate
\begin{align}
|u(x)-P(x)|  &  \leq|u(x)-P_{i}(x)|+|P_{i}(x)-P(x)|\nonumber\\
&  \leq MC_{0}r_{0}^{i(2+\bar{\alpha})}+\frac{3MC_{0}}{1-r_{0}^{\bar{\alpha}}%
}r_{0}^{i(2+\bar{\alpha})}\nonumber\\
&  \leq MC_{0}^{\prime}\left\Vert x\right\Vert ^{2+\bar{\alpha}}
\label{estimate4}%
\end{align}
where
\begin{equation}
C_{0}^{\prime}=C_{0}(1+\frac{3}{1-r_{0}^{\bar{\alpha}}})\frac{1}{r_{0}%
^{1+\bar{\alpha}}}. \label{cnotprime}%
\end{equation}
This completes the proof of (\ref{e1}).\newline Next, consider any point
$x_{0}$ in $B_{1/2}$. Let $v(x^{\prime})=4u(x^{\prime}/2+x_{0})$ where $x\in
B_{1}$. Note that $B_{1/2}(x_{0})\subset B_{1}$ and hence $F(D^{2}v)=0$ makes
sense on $B_{1}$. Applying estimate (\ref{estimate4}) to $v$ for $x^{\prime
}=2(x-x_{0})~$yields a polynomial $P_{x_{0}}(x)$ such that
\[
|u(x)-P_{x_{0}}(x)|\leq MC_{0}^{\prime}2^{\bar{\alpha}}\left\Vert
x-x_{0}\right\Vert ^{2+\bar{\alpha}}%
\]
holds on $B_{1/2}(x_{0}).$

The following Lemma has been used in passing in the literature~ \cite[Remark
3, page 74]{CC}. We state it here for precision in the estimate. For the proof
see Corollary \ref{CC81} in Appendix 1.

\begin{lemma}
\label{cc8}Suppose for all $x_{0}\in B_{1/2}$ there a second order polynomial
$P_{x_{0}}$ such that%
\[
|u(x)-P_{x_{0}}(x)|\leq K\left\Vert x-x_{0}\right\Vert ^{2+\bar{\alpha}}%
\]
and
\[
\left\vert P_{x_{0}}\right\vert \leq K
\]

on $B_{1}.$ Then $\left\Vert D^{2}u\right\Vert _{C^{\bar{\alpha}}\left(
B_{1/4}\right)  }\leq\left(  2+2^{2+\bar{\alpha}}\right)  ^{2}K.$
\end{lemma}

It follows from Lemma \ref{cc8} that $u\in C^{2,\bar{\alpha}}(\bar{B}%
_{1/4}(0))$ with bounds given by
\begin{equation}
||D^{2}u||_{C^{\bar{\alpha}}(\bar{B}_{1/4}(0))}\leq C_{0}^{\prime}%
2^{\bar{\alpha} }\left(  2+2^{2+\bar{\alpha}}\right)  ^{2}||u||_{L^{\infty
}(B_{1})}. \label{C1}%
\end{equation}
Combining (\ref{cnotprime}) with (\ref{C1}) proves the estimate (\ref{eee}).
\end{proof}

\begin{proof}
[Proof of Theorem \ref{Theorem 1.2}]: \newline

We are assuming that $F$ is an operator on the space of symmetric matrices,
and thus we can take a $DF$ that is symmetric. Let
\[
W=DF(\mathbf{0})
\]
which will be a positive symmetric matrix, by ellipticity. In particular
\[
\lambda I\leq W\leq\Lambda I.
\]
We can find a positive square root of the inverse, namely
\begin{equation}
AA^{T}=W^{-1}.\label{ataw}%
\end{equation}
Now define%
\[
\tilde{F}(N)=F(ANA^{T}).
\]
Observe
\begin{align*}
\frac{\partial\tilde{F}}{\partial n_{ij}}|_{N=\mathbf{0}\ } &  =\frac{\partial
F}{\partial a_{pq}}|_{\mathbf{0}}\frac{\partial(ANA^{T})_{pq}}{\partial
n_{ij}}\\
&  =\sum_{p,q}\frac{\partial F}{\partial a_{pq}}|_{\mathbf{0}}A_{pi}A_{jq}%
^{T}\\
&  =W_{pq}A_{pi}A_{jq}^{T}=\left(  A^{T}WA\right)  _{ij}.
\end{align*}
But by (\ref{ataw}),
\[
A^{T}WA=I.
\]
It follows that $D\tilde{F}(\mathbf{0})=I$. Note that $\tilde{F}$ has
ellipticity constants in $[\frac{\lambda}{\Lambda},\frac{\Lambda}{\lambda}]$.

Finally, note that if $F$ satisfies a $\varepsilon_{0}$ closeness condition
then
\begin{align*}
\left\Vert D\tilde{F}(M)-D\tilde{F}(N)\right\Vert  &  =\left\Vert
A^{T}DF(AMA^{T})A-A^{T}DF(ANA^{t})A\right\Vert \\
&  =\left\Vert A^{T}\left(  DF(AMA^{T})-DF(ANA^{T})\right)  A\right\Vert \\
&  \leq\varepsilon_{0}\left\Vert A^{T}A\right\Vert \leq\varepsilon_{0}\Lambda.
\end{align*}
Therefore, $\tilde{F}$ is almost linear with constant $\varepsilon_{0}%
\Lambda.$

Now we let
\begin{equation}
\varepsilon_{0}\left(  n,\lambda,\Lambda,\bar{\alpha}\right)  :=\frac
{1}{\Lambda}\tilde{\varepsilon}_{0}(n,\frac{\lambda}{\Lambda},\frac{\Lambda
}{\lambda},\bar{\alpha})\label{eps_0}%
\end{equation}
for $\tilde{\varepsilon_{0}}$ defined in Lemma \ref{lemma1}. It follows that
$\tilde{F}$ satisfies the $\tilde{\varepsilon}_{0}$ criterion of \ref{lemma1}
when $F$ satisfies the $\varepsilon_{0}$ closeness condition. Now let%
\[
v(x)=u(\left(  A^{-1}\right)  ^{T}x).
\]
Notice that
\begin{align*}
v_{ij}(x) &  =u_{kl}A_{jk}^{-1}A_{il}^{-1}\\
D^{2}v &  =A^{-1}D^{2}u\left(  \left(  A^{-1}\right)  ^{T}x\right)  \left(
A^{-1}\right)  ^{T}%
\end{align*}
so
\begin{align*}
\tilde{F}(N) &  =F(AA^{-1}D^{2}u\left(  \left(  A^{-1}\right)  ^{T}x\right)
\left(  A^{-1}\right)  ^{T}A^{T})\\
&  =F\left(  D^{2}u\left(  \left(  A^{-1}\right)  ^{T}x\right)  \right)  =0.
\end{align*}
Now if $u$ is defined on $B_{1}$, the new function $v$ will be defined on
$B_{\frac{1}{\sqrt{\Lambda}}}.$ Rescaling
\[
\tilde{v}=\Lambda v\left(  \frac{x}{\sqrt{\Lambda}}\right)
\]
we have a function defined on $B_{1}$ and can apply Proposition \ref{prop} to
$\tilde{v}:$%
\[
||D^{2}\tilde{v}||_{C^{\bar{\alpha}}(B_{1/4})}\leq\tilde{C}_{1}||\tilde
{v}||_{L^{\infty}(B_{1})}\leq\Lambda\tilde{C}_{1}||u||_{L^{\infty}(B_{1})}.
\]
Meanwhile, provided that%
\begin{align*}
\sqrt{\Lambda}A^{T}x &  \in B_{1/4}\\
\sqrt{\Lambda}A^{T}y &  \in B_{1/4},
\end{align*}
we have%
\begin{align*}
\frac{\left\Vert D^{2}u(x)-D^{2}u(y)\right\Vert }{\left\vert x-y\right\vert
^{\bar{\alpha}}} &  =\frac{\left\Vert A\left(  D^{2}v(A^{T}x)-D^{2}%
v(A^{T}y)\right)  A^{T}\right\Vert }{\left\vert x-y\right\vert ^{\bar{\alpha}%
}}\\
&  =\frac{\left\Vert A\left(  D^{2}\tilde{v}(\sqrt{\Lambda}A^{T}x)-D^{2}%
\tilde{v}(\sqrt{\Lambda}A^{T}y)\right)  A^{T}\right\Vert }{\left\vert
x-y\right\vert ^{\bar{\alpha}}}\\
&  \leq\frac{\Lambda}{\left\vert x-y\right\vert ^{\bar{\alpha}}}\left\Vert
D^{2}\tilde{v}\right\Vert _{C^{\bar{\alpha}}(B_{1/4})}\left\vert \sqrt
{\Lambda}A^{T}x-\sqrt{\Lambda}A^{T}y\right\vert ^{\bar{\alpha}}\\
&  \leq\Lambda^{1+\bar{\alpha}}\left\Vert D^{2}\tilde{v}\right\Vert
_{C^{\bar{\bar{\alpha}}}(B_{1/4})}\\
&  \leq\Lambda^{2+\bar{\alpha}}\tilde{C}_{1}||u||_{L^{\infty}(B_{1})}.
\end{align*}
We conclude that for $x\in B_{1/4\Lambda}$ the estimate holds.
\end{proof}

\section{Proof of Theorem \ref{Theorem 1.3}}

To begin proving Theorem \ref{Theorem 1.3} we require the following version of
\cite[Lemma 7.9]{CC}:

\begin{lemma}
\label{lemma 2.4} Let $u$ be a viscosity solution of (\ref{F}) in $B_{4/7}$
such that $||u||_{L^{\infty}(B_{4/7})}\leq1$ and $f\in L^{n}(B_{4/7})$. Assume
that $F(D^{2}w)=0$ has $C^{1,1}$ interior estimates (with constant $C_{1}$).
Then there exists a function $h\in C^{2}(\bar{B}_{3/7})$ such that $h$
satisfies $||h||_{C^{1,1}(\bar{B}_{3/7})}\leq c(n)C_{1}$ (for a constant
\textsubscript {}$c(n)$ depending only on $n$) and
\begin{align}
||u-h||_{L^{\infty}(B_{3/7})}  &  \leq C_{3}||f||_{L^{n}(B_{4/7})}%
\label{L7.9}\\
F(D^{2}h)  &  =0\hspace{0.2cm}in\hspace{0.1cm}B_{1/2}\nonumber\\
h  &  =u\hspace{0.4cm}on\hspace{0.1cm}\partial{B_{1/2}}.\nonumber
\end{align}
Here $C_{3}$ is a positive constant depending on $n,\lambda,\Lambda,C_{1}$.
\end{lemma}

\textbf{Note:} We say that $F(D^{2}w) =0$ has $C^{1 ,1}$ interior estimates
(with constant $C_{1}$) if for any $w_{0}\in C(\partial B)$ there exists a
solution $w\in C^{2}(B_{1})\cap C(\bar{B}_{1})$ of
\begin{align*}
F(D^{2}w)=0 \hspace{1cm}in\hspace{.2cm}B_{1}\\
w=w_{0}\hspace{.5cm}on\hspace{.2cm}\partial B_{1}%
\end{align*}
such that $\vert\vert w\vert\vert_{C^{1 ,1}(\bar{B}_{1/2})} \leq C_{1}%
||w_{0}||_{L^{\infty}(\partial B_{1})} $.

\begin{proof}
The statement in \cite[lemma 7.9]{CC} is given for elliptic operators
$F(D^{2}w,x)$ that may depend also on $x.$ The obvious approximation argument
when there is no dependence on $x$ gives the proof of Lemma \ref{lemma 2.4}.
\end{proof}

\begin{lemma}
\label{lemma 2.6}There exists $\delta>0$ depending on $n,\lambda,\Lambda,$ and
$\alpha<\bar{\alpha}$ such that if $u$ is a viscosity solution of (\ref{F}) in
$B_{1}$ with $F$ almost linear with constant $\varepsilon_{0}(n,\lambda
,\Lambda,\bar{\alpha})$ with
\[
||u||_{L^{\infty}(B_{1})}\leq1
\]
and
\begin{equation}
\left(  \frac{1}{|B_{r}|}\int_{B_{r}}|f|^{n}\right)  ^{1/n}\leq\delta
r^{\alpha}\hspace{0.2cm}\forall r\leq1 \label{DD}%
\end{equation}
then there exists a polynomial $P$ of degree 2 such that
\begin{align}
||u-P||_{L^{\infty}(B_{r})}  &  \leq C_{4}r^{2+\alpha}\hspace{0.2cm}\forall
r\leq1,\nonumber\\
|DP(0)|+||D^{2}P||  &  \leq C_{4} \label{P1}%
\end{align}
for some constant $C_{4}>0$ depending only on $n,\lambda,\Lambda,\alpha$.
\end{lemma}

\begin{proof}
The proof follows from the following claim.

\begin{claim}
\label{claim_CL3} Given $\lambda,\Lambda,$ and $\alpha<\bar{\alpha},$ suppose
that $u$ is a viscosity solution of (\ref{F}) in $B_{1}$ for $F$ almost linear
with constant $\varepsilon_{0}(n,\lambda,\Lambda,\bar{\alpha})$, with $f$
satisfying (\ref{DD}) and $u~$satisfying
\begin{equation}
||u||_{L^{\infty}(B_{1})}\leq1. \label{L0}%
\end{equation}
Then there exists $\delta>0$, $0<\mu<1$ and a sequence%
\[
P_{k}(x)=a_{k}+b_{k}\cdot x+\frac{1}{2}x^{t}c_{k}\cdot x
\]
satisfying
\begin{align}
F(D^{2}P_{k})  &  =0\label{L1b}\\
||u-P_{k}||_{L^{\infty}(B_{\mu^{k}})}  &  \leq\mu^{k(2+\alpha)}\label{L1a}\\
|a_{k}-a_{k-1}|+\mu^{k-1}|b_{k}-b_{k-1}|+\mu^{2(k-1)}|c_{k}-c_{k-1}|  &  \leq
C_{1}\mu^{(k-1)(2+\alpha)}. \label{L1}%
\end{align}

\end{claim}

We first prove the claim.

\begin{proof}
Let $P_{0}=0$. Then for $k=0$, we see that (\ref{L1a}) holds trivially for any
$\mu>0$ from (\ref{L0})$.$ For $\mu$ determined by (\ref{L2}), we will show
that whenever (\ref{L1a}) holds for $k=i$, then there exist $P_{i+1\text{ }}$
so that (\ref{L1a}) holds for $k=i+1.$

We choose $\mu$ small enough such that
\begin{equation}
2C_{1}\mu^{\bar{\alpha}}\leq\mu^{\alpha} \label{L2}%
\end{equation}
and
\begin{equation}
\mu^{\alpha}\leq3/7. \label{L2B}%
\end{equation}
We define
\begin{align}
v_{i}(x)  &  =\frac{(u-P_{i})(\mu^{i}x)}{\mu^{i(2+\alpha)}},\label{vedf}\\
F_{i}(N)  &  =\frac{F(\mu^{i\alpha}N+c_{i})}{\mu^{i\alpha}},\nonumber\\
f_{i}(x)  &  =\frac{f(\mu^{i}x)}{\mu^{i\alpha}},\nonumber
\end{align}
where $P_{i}(x)=a_{i}+\vec{b}_{i}\cdot x+\frac{1}{2}x^{T}\cdot\mathbf{c}_{i}%
x$. Thus
\begin{equation}
F_{i}(D^{2}v_{i}(x))=f_{i}(x). \label{V}%
\end{equation}
Note that
\begin{equation}
\left\Vert v_{i}\right\Vert _{L^{\infty}(B_{1})}\leq1 \label{vleq1}%
\end{equation}
by (\ref{L1a}). Now we choose $\delta$ small enough such that
\begin{equation}
\omega_{n}^{1/n}C_{3}\delta\leq C_{1}\mu^{2+\bar{\alpha}} \label{GG}%
\end{equation}
where $\omega_{n}$ is the volume of a unit ball in $n$ dimensions and $C_{3}$
is the constant appearing in the first inequality of (\ref{L7.9}) in Lemma
\ref{lemma 2.4}.

We consider the equation (\ref{V}). Observe that (\ref{DD}) implies
\begin{align}
||f_{i}||_{L^{n}(B_{1})}  &  =\mu^{-i\alpha}\mu^{-i}||f||_{L^{n}(B_{\mu^{i}}%
)}\label{VV}\\
&  \leq\mu^{-i\alpha}\mu^{-i}\left\vert B_{\mu^{i}}\right\vert ^{1/n}\delta
\mu^{i\alpha}=\left(  \omega_{n}\right)  ^{1/n}\delta.
\end{align}
Note that $F_{i}$ satisfies
\[
F_{i}(0)=\frac{F(c_{i})}{\mu^{i\alpha}}=\frac{F(D^{2}P_{i})}{\mu^{i\alpha}}=0
\]
and
\[
DF_{i}(N)=DF(\mu^{i\alpha}N+c_{i})
\]
so $F_{i}$ also satisfies the $\varepsilon_{0}(n,\lambda,\Lambda,\bar{\alpha
})$ closeness condition (\ref{Cm1}) when $F$ does. Since $||v_{i}%
||_{L^{\infty}(B_{1})}\leq1$, by applying Lemma \ref{lemma 2.4} to (\ref{V})
considering (\ref{VV}) we see that there exists $h\in C^{2}(\bar{B}_{3/7})$
such that
\begin{equation}
||v_{i}-h||_{L^{\infty}(B_{3/7})}\leq\omega_{n}^{1/n}C_{3}\delta\label{L4}%
\end{equation}
and $h$ solves the following boundary value problem:
\begin{align}
F_{i}(D^{2}h)  &  =0\hspace{0.2cm}in\hspace{0.1cm}B_{1/2}\nonumber\\
h  &  =v_{i}\hspace{0.4cm}on\hspace{0.1cm}\partial{B_{1/2}}. \label{hcondv}%
\end{align}
Then from the definition of $F_{i}$ above, it follows that
\begin{equation}
F(\mu^{i\alpha}D^{2}h+c_{i})=0\hspace{0.2cm}in\hspace{0.2cm}B_{1/2}.
\label{L5}%
\end{equation}
Now apply Theorem \ref{Theorem 1.2} to $h$ so see that
\begin{align}
||h||_{C^{2,\bar{\alpha}}(B_{1/4})}  &  \leq C_{1}||h||_{L^{\infty}(\partial
B_{1/2})}\label{H1}\\
&  \leq C_{1}||v_{i}||_{L^{\infty}(\partial B_{1/2})}\\
&  \leq C_{1} \label{H1B}%
\end{align}
from (\ref{hcondv}) and the maximum principle (cf. \cite[Proposition 2.13]%
{CC}), and the last inequality follows from (\ref{vleq1}). Since $h$ is
$C^{2,\bar{\alpha}}$, there exists a polynomial $\bar{P}$ given by
\[
\bar{P}(x)=h(0)+Dh(0)\cdot x+\frac{1}{2}x^{t}D^{2}h(0)\cdot x
\]
such that
\begin{equation}
||h-\bar{P}||_{L^{\infty}(B_{\mu})}\leq C_{1}\mu^{2+\bar{\alpha}}. \label{L7}%
\end{equation}

From (\ref{L4}), (\ref{L2B}) and (\ref{L7}) we have
\begin{align}
||v_{i}-\bar{P}||_{L^{\infty}(B_{\mu})}  &  \leq||v_{i}-h||_{L^{\infty}%
(B_{\mu})}+||h-\bar{P}||_{L^{\infty}(B_{\mu})}\nonumber\\
&  \leq\omega_{n}^{1/n}C_{3}\delta+C_{1}\mu^{2+\bar{\alpha}}\nonumber\\
&  \leq2C_{1}\mu^{2+\bar{\alpha}}\nonumber\\
&  \leq\mu^{2+\alpha} \label{L8}%
\end{align}
where the last two inequalities follow from (\ref{GG}) and (\ref{L2}).

Rescaling the bound (\ref{L8}) back via (\ref{vedf}) we see that
\begin{equation}
|u(x)-P_{i}(x)-\mu^{i(2+\alpha)}\bar{P}(\mu^{-i}x)|\leq\mu^{(2+\alpha)(i+1)}
\label{L9}%
\end{equation}
for all $x\in B_{\mu^{i+1}}.$

We define
\begin{equation}
P_{i+1}(x)=P_{i}(x)+\mu^{i(2+\alpha)}\bar{P}(\mu^{-i}x) \label{H2}%
\end{equation}
and we have
\begin{equation}
\mathbf{c}_{i+1}=\mathbf{c}_{i}+\mu^{i\alpha}D^{2}h(0). \label{L5s}%
\end{equation}
From (\ref{L9}) we see that
\[
||u-P_{i+1}||_{L^{\infty}(B_{\mu^{i+1}})}\leq\mu^{(i+1)(2+\alpha)}%
\]
which proves (\ref{L1a}) for $k=i+1.$ Now from (\ref{L5}) and (\ref{L5s}) we
get
\[
F(\mathbf{c}_{i+1})=0
\]
proving (\ref{L1b}). Now evaluating (\ref{H2}) and its first and second
derivates at $x=0$ yields
\begin{align*}
a_{i+1}-a_{i}  &  =\mu^{i(2+\alpha)}\bar{P}(0)\\
\vec{b}_{i+1}-\vec{b}_{i}  &  =\mu^{i(1+\alpha)}D\bar{P}(0)\\
\mathbf{c}_{i+1}-\mathbf{c}_{i}  &  =\mu^{i\alpha}D^{2}\bar{P}(0).
\end{align*}
Thus
\begin{align*}
&  |a_{i+1}-a_{i}|+\mu^{i}\left\Vert \vec{b}_{i+1}-\vec{b}_{i}\right\Vert
+\mu^{2i}\left\Vert \mathbf{c}_{i+1}-\mathbf{c}_{i}\right\Vert \\
&  =\mu^{i(2+\alpha)}(|h(0)|+\left\Vert Dh(0)\right\Vert +\left\Vert
D^{2}h(0)\right\Vert )\\
&  \leq\mu^{i(2+\alpha)}C_{1}%
\end{align*}
by (\ref{H1B}), proving (\ref{L1}). This proves claim \ref{claim_CL3}.
\end{proof}

Now we return to proving Lemma \ref{lemma 2.6}, which will follow by arguments
similar to those used in the proof of Theorem \ref{Theorem 1.2} following
(\ref{Pmarker}). In particular, define%
\[
P=\lim_{i\rightarrow\infty}P_{i}=\sum_{i=0}^{\infty}(P_{i+1}-P_{i})
\]
which will have coefficients%
\begin{align*}
a  &  =\sum_{i=0}^{\infty}(a_{i+1}-a_{i})\\
\vec{b}  &  =\sum_{i=0}^{\infty}(\vec{b}_{i+1}-\vec{b}_{i})\\
\mathbf{c}  &  =\sum_{i=0}^{\infty}(\mathbf{c}_{i+1}-\mathbf{c}_{i}).
\end{align*}
Note that by (\ref{L1})
\begin{align*}
\left\vert a_{i+1}-a_{i}\right\vert  &  \leq C_{1}\mu^{i(2+\alpha)}\\
\left\Vert \vec{b}_{i+1}-\vec{b}_{i}\right\Vert  &  \leq C_{1}\mu
^{i(1+\alpha)}\\
\left\Vert \mathbf{c}_{i+1}-\mathbf{c}_{i}\right\Vert  &  \leq C_{1}%
\mu^{i\alpha}.
\end{align*}
We conclude that the tails of the constant, linear, and quadratic terms of the
polynomial series converge uniformly with upper bounds given by
\begin{align*}
\left\vert \sum_{j=i}^{\infty}(a_{j+1}-a_{j})\right\vert  &  \leq C_{1}%
\mu^{i(2+\alpha)}\frac{1}{1-\mu^{(2+\alpha)}}\\
\left\vert \sum_{j=i}^{\infty}(\vec{b}_{j+1}-\vec{b}_{j})\right\vert  &  \leq
C_{1}\mu^{i(1+\alpha)}\frac{1}{1-\mu^{(1+\alpha)}}\\
\left\vert \sum_{j=i}^{\infty}(\mathbf{c}_{j+1}-\mathbf{c}_{j})\right\vert  &
\leq C_{1}\mu^{i\alpha}\frac{1}{1-\mu^{\alpha}}%
\end{align*}
respectively. Thus $P$ is well-defined. Next,
\begin{align*}
||u-P||_{L^{\infty}(B_{\mu^{i}})}  &  \leq||u-P_{i}||_{L^{\infty}(B_{\mu^{i}%
})}+\sum_{j=i}^{\infty}||P_{j+1}-P_{j}||_{L^{\infty}(B_{\mu^{i}})}\\
&  \leq\mu^{i(2+\alpha)}+\sum_{j=i}^{\infty}[|a_{j+1}-a_{j}|+\mu^{i}\left\Vert
\vec{b}_{j+1}-\vec{b}_{j}\right\Vert +\frac{1}{2}\mu^{2i}\left\Vert
\mathbf{c}_{j+1}-\mathbf{c}_{j}\right\Vert ]\\
&  \leq\mu^{i(2+\alpha)}+C_{1}\left\{
\begin{array}
[c]{c}%
\mu^{i(2+\alpha)}\frac{1}{1-\mu^{(2+\alpha)}}\\
+\mu^{i}\mu^{i(1+\alpha)}\frac{1}{1-\mu^{(1+\alpha)}}\\
+\mu^{2i}\mu^{i\alpha}\frac{1}{1-\mu^{\alpha}}%
\end{array}
\right\} \\
&  \leq C_{4}\mu^{i(2+\alpha)}%
\end{align*}
where
\[
C_{4}=1+C_{1}\left(  \frac{1}{1-\mu^{\alpha}}+\frac{1}{1-\mu^{\alpha}}%
+\frac{1}{1-\mu^{\alpha}}\right)  .
\]
Clearly we have
\[
|DP(0)|+||D^{2}P||\leq C_{1}\frac{1}{1-\mu^{1+\alpha}}+C_{1}\frac{1}%
{1-\mu^{\alpha}}.
\]
We see that (\ref{P1}) holds good for $C_{4}$. This proves the lemma.
\end{proof}

\begin{proof}
[Proof of Theorem \ref{Theorem 1.3}]Fix $\alpha<\bar{\alpha}$. We will first
prove that the estimate (\ref{AAA}) holds at the origin, in particular, we
show that there exists a polynomial of degree 2 such that
\begin{align}
||u-P||_{L^{\infty}(B_{r})}  &  \leq C_{2}^{\prime}r^{2+\alpha}\hspace
{0.2cm},\forall r\leq1\nonumber\\
|DP(0)|+||D^{2}P||  &  \leq C_{2}^{\prime} \label{P2}%
\end{align}
where $C_{2}^{\prime}=C_{2}^{\prime}(||u||_{L^{\infty}(B_{1})},|f|_{C^{\alpha
}(B_{1})},n,\lambda,\Lambda,\bar{\alpha},\alpha,C_{1})$, $0<\alpha<\bar
{\alpha}$ and $\bar{\alpha}$ is the H\"{o}lder exponent appearing in
(\ref{EE}):%
\[
||u||_{C^{2,\bar{\alpha}}(B_{1/2})}\leq C_{1}||u||_{L^{\infty}(B_{1})}.
\]

Let
\[
\tilde{f}(x)=f(x)-f(0)
\]
so that the $C^{\alpha}$ function $\tilde{f}(x)$ satisfies the following
\[
(\frac{1}{|B_{1}|}\int_{B_{1}}|\tilde{f}|^{n})^{1/n}\leq\left\Vert \tilde
{f}\right\Vert _{C^{\alpha}(B_{1})}.
\]

The proof now follows directly from Lemma \ref{lemma 2.6}, if we do the
following rescaling for all $x\in B_{1}$: Consider the following function
\[
\tilde{u}(x)=\frac{u(x)}{\delta^{-1}|f|_{C^{\alpha}(B_{1})}+||u||_{L^{\infty
}(B_{1})}}=\frac{u(x)}{T}.
\]
with $\delta(n,\lambda,\Lambda,\alpha,\bar{\alpha})$ as defined in (\ref{GG}).
Note that
\begin{align}
\delta T  &  =\left\Vert \tilde{f}\right\Vert _{C^{\alpha}(B_{1})}%
+\delta||u||_{L^{\infty}(B_{1})}\nonumber\\
&  >\left\Vert \tilde{f}\right\Vert _{C^{\alpha}(B_{1})} \label{RR}%
\end{align}
and that
\[
||\tilde{u}||_{L^{\infty}(B_{1})}\leq1.
\]
Now we consider the operator
\[
F_{T}(N)=\frac{1}{T}F(TN)
\]
defined for all $N\in S_{n}$.\newline Note that $F_{T}$ satisfies the
following properties:

\begin{enumerate}
[label=(\roman*)]

\item $F_{T}$ has the same ellipticity constants $\lambda$ and $\Lambda$ as
$F$.

\item $DF_{T}$ satisfies condition (\ref{Cm1}) with the same constant
$\varepsilon_{0}\left(  n,\lambda,\Lambda,\bar{\alpha}\right)  $ if $DF$ does.
\end{enumerate}

We see that $\tilde{u}$ satisfies the equation%
\[
F_{T}(D^{2}\tilde{u}(x))=\frac{1}{T}F(TD^{2}\tilde{u}(x))=\frac{1}{T}%
F(D^{2}u(x))=\frac{\tilde{f}(x)}{T}=f_{T}(x),
\]
where for $r\leq1$ we compute
\begin{align*}
\left(  \frac{1}{|B_{r}|}\int_{B_{r}}|f_{T}|^{n}\right)  ^{1/n}  &  \leq
\frac{\left\Vert \tilde{f}\right\Vert _{C^{\alpha}(B_{r})}}{T}\left(  \frac
{1}{1+\alpha}\right)  ^{1/n}r^{\alpha}\\
&  <\delta r^{\alpha}%
\end{align*}
recalling (\ref{RR}) in the last inequality.

Therefore, the equation
\[
F_{T}(D^{2}\tilde{u}(x))=f_{T}(x)
\]
satisfies all the conditions of Lemma \ref{lemma 2.6} and hence the function
$\tilde{u}$ satisfies the estimates (\ref{P1}). In particular, there exists
$\tilde{P}$ such that \newline%
\begin{align}
||\tilde{u}-\tilde{P}||_{L^{\infty}(B_{r})}  &  \leq C_{4}r^{2+\alpha},\text{
}\hspace{0.2cm}\forall r\leq1,\label{fP}\\
|D\tilde{P}(0)|+||D^{2}\tilde{P}||  &  \leq C_{4}(n,\lambda,\Lambda,\alpha)
\label{fP1}%
\end{align}
that is, letting
\[
P=\left(  \delta^{-1}|f|_{C^{\alpha}(B_{1})}+||u||_{L^{\infty}(B_{1})}\right)
\tilde{P}%
\]
we have
\begin{align}
||u-P||_{L^{\infty}(B_{r})}  &  \leq C_{4}r^{2+\alpha}\hspace{0.2cm}\forall
r\leq1,\nonumber\\
|DP(0)|+||D^{2}P||  &  \leq\left(  \delta^{-1}|f|_{C^{\alpha}(B_{1}%
)}+||u||_{L^{\infty}(B_{1})}\right)  C_{4}(n,\lambda,\Lambda,\alpha).
\end{align}
Next, consider any point $x_{0}$ in $B_{1/2}$. The remainder of the proof
follows verbatim from the argument following (\ref{cnotprime}).
\end{proof}

\section{Appendix 1:Pointwise H\"{o}lder implies H\"{o}lder}

\begin{lemma}
Suppose that
\[
U:B_{R}(0)\subset%
\mathbb{R}
^{n}\rightarrow%
\mathbb{R}
\]
satsifies the following condition for some fixed $p>0$. For every $y$ there
exists a linear function $L_{y}$ such that
\begin{equation}
\left\vert U(x)-L_{y}(x)\right\vert \leq C_{1}|x-y|^{1+p}. \label{cond}%
\end{equation}
Then for all $x\in B_{R/2}(0)$ we have%
\[
\left\vert DU(x)-DU(0)\right\vert \leq C_{1}\left(  2+2^{1+p}\right)
\left\vert x\right\vert ^{p}.
\]

\end{lemma}

\begin{proof}
We will assume by adding a linear function that
\begin{align}
U(0)  &  =0\label{b1}\\
DU(0)  &  =0.\nonumber
\end{align}
First note that (\ref{cond}) implies that the derivative exists at any $x_{0}$
and
\[
L_{x_{0}}=DU(x_{0})\cdot(x-x_{0})+U(x_{0})
\]
thus
\[
\left\vert U(x)-DU(x_{0})\cdot(x-x_{0})-U(x_{0})\right\vert \leq C_{1}%
|x-x_{0}|^{1+p}.
\]
That is%
\begin{equation}
\left\vert DU(x_{0})\cdot(x-x_{0})\right\vert \leq C_{1}|x-x_{0}%
|^{1+p}+\left\vert U(x)\right\vert +\left\vert U(x_{0})\right\vert .
\label{b2}%
\end{equation}
Now consider any point $x_{0}\neq0$ with $x_{0}\in B_{R/2}.$ Let
\begin{equation}
A=DU(x_{0}) \label{defA}%
\end{equation}
and let
\[
e=\frac{A}{\left\Vert A\right\Vert }.
\]
Now consider the point
\[
x_{1}=\left(  x_{0}+|x_{0}|\frac{A}{\left\Vert A\right\Vert }\right)
\]
which satisfies%
\[
\left\vert x_{1}\right\vert \leq2|x_{0}|.
\]
So $x_{1}\in B_{R}.$ Letting $y=0$ in (\ref{cond}) and using (\ref{b1}) we
conclude
\begin{equation}
\left\vert U(x_{1})\right\vert \leq C_{1}2^{1+p}|x_{0}|^{1+p}. \label{b3}%
\end{equation}
Plugging $x_{1}$ into (\ref{b2}) and using (\ref{b3})
\begin{align}
\left\vert DU(x_{0})\cdot(x_{1}-x_{0})\right\vert  &  \leq C_{1}|x_{1}%
-x_{0}|^{1+p}+\left\vert U(x_{1})\right\vert +\left\vert U(x_{0})\right\vert
\\
&  \leq C_{1}|x_{1}-x_{0}|^{1+p}+C_{1}2^{1+p}|x_{0}|^{1+p}+C_{1}|x_{0}|^{1+p}.
\end{align}
But%
\[
x_{1}-x_{0}=x_{0}+|x_{0}|\frac{A}{\left\Vert A\right\Vert }-x_{0}=|x_{0}%
|\frac{A}{\left\Vert A\right\Vert }=|x_{0}|\frac{DU(x_{0})}{\left\Vert
DU(x_{0})\right\Vert }%
\]
and
\[
\left\vert x_{1}-x_{0}\right\vert =\left\vert |x_{0}|\frac{DU(x_{0}%
)}{\left\Vert DU(x_{0})\right\Vert }\right\vert =|x_{0}|.
\]
So we have shown that
\begin{equation}
|x_{0}|\left\Vert DU(x_{0})\right\Vert \leq C_{1}|x_{0}|^{1+p}+C_{1}%
2^{1+p}|x_{0}|^{1+p}+C_{1}|x_{0}|^{1+p}%
\end{equation}
that is
\[
\left\Vert DU(x_{0})\right\Vert \leq C_{1}\left(  2+2^{1+p}\right)
|x_{1}-x_{0}|^{p}.
\]

\end{proof}

\begin{corollary}
\label{CC81}Suppose that
\[
u:B_{1}(0)\subset%
\mathbb{R}
^{n}\rightarrow%
\mathbb{R}
\]
satisfies the following condition for some fixed $p>0$. For every $y\in
B_{1/2}$ there exists a quadratic function $Q_{y}$ such that
\begin{equation}
\left\vert U(x)-Q_{y}(x)\right\vert \leq C_{1}|x-y|^{2+\alpha}.
\end{equation}
Then for $x\in B_{1/4}(0)$%
\[
\sup_{i,j}\left\vert u_{ij}(x)-u_{ij}(0)\right\vert \leq\left(  2+2^{2+\alpha
}\right)  ^{2}C_{1}\left\vert x\right\vert ^{\alpha}.
\]

\end{corollary}

\begin{proof}
As before subtract off a quadratic function so that $u$ vanishes at secord
order at $0.$ Apply the previous Lemma with $p=1+\alpha$ and conclude that for
all $x\in B_{1/2}$%
\[
\left\vert DU(x)-DU(0)\right\vert \leq C_{1}\left(  2+2^{2+\alpha}\right)
\left\vert x\right\vert ^{1+\alpha}%
\]
that is
\[
\left\vert u_{i}(x)\right\vert \leq\left(  2+2^{2+\alpha}\right)
C_{1}|x|^{1+\alpha}.
\]
So we apply the previous Lemma, with $R=1/2$ and conclude that
\[
\left\vert Du_{i}(x)\right\vert \leq\left(  2+2^{2+\alpha}\right)  ^{2}%
C_{1}|x|^{\alpha}.
\]

\end{proof}

\section{Appendix 2: Cordes-Nirenberg\label{cnsection}}

In \cite[Lemma 3]{N53}, Nirenberg proved the following result (slightly reworded).

\begin{lemma}
\label{nlemma}Let $U=(u_{1},u_{2})$ be an $%
\mathbb{R}
^{2}$-valued continuous function defined in a domain $B_{1}\subset$ $%
\mathbb{R}
^{2}$ having continuous first derivatives satisfying
\begin{equation}
u_{1,1}^{2}+u_{1,2}^{2}+u_{2,1}^{2}+u_{2,2}^{2}\leq k\left(  u_{1,2}%
u_{2,1}-u_{1,1}u_{2,2}\right)  +k_{1} \label{ncond}%
\end{equation}
and let $d<1$.

Then there exists $M,\alpha$ depending on $k,k_{1}$, and $d$ such that
\[
\int\int_{B_{d}(0)}r^{-\alpha}\left(  u_{1,1}^{2}+u_{1,2}^{2}+u_{2,1}%
^{2}+u_{2,2}^{2}\right)  dxdy\leq M.
\]

\end{lemma}

With this integral estimate in hand, a univeral H\"{o}lder estimate on the
functions $u_{1}$ and $u_{2}$ follows.

Now suppose that
\begin{equation}
a^{ij}u_{ij}=f. \label{linf}%
\end{equation}
Note
\[
\Delta u=\left(  \delta^{ij}-a^{ij}\right)  u_{ij}+f
\]
which implies
\[
\left\vert \Delta u\right\vert \leq\left\Vert \delta^{ij}-a^{ij}\right\Vert
_{HS}\left\Vert u_{ij}\right\Vert _{HS}+f.
\]
In two dimensions, we have
\[
\left\Vert D^{2}u\right\Vert ^{2}=\left(  \Delta u\right)  ^{2}-2\det\left(
D^{2}u\right)
\]
so%
\[
\left\Vert D^{2}u\right\Vert ^{2}\leq(1+\varepsilon)\left\Vert \delta
^{ij}-a^{ij}\right\Vert _{HS}^{2}\left\Vert D^{2}u\right\Vert _{HS}%
^{2}+(1+\frac{1}{\varepsilon})f^{2}+2\left(  u_{1,2}u_{2,1}-u_{1,1}%
u_{2,2}\right)  .
\]
In particular, (\ref{ncond}) holds with constants
\begin{align*}
k  &  =\frac{2}{1-(1+\varepsilon)\left\Vert \delta^{ij}-a^{ij}\right\Vert
_{HS}^{2}}\\
k_{1}  &  =(1+\frac{1}{\varepsilon})\left\Vert f\right\Vert _{L^{\infty}}.
\end{align*}
Thus in two dimensions a $C^{1,\alpha}$ estimate is available provided
\[
\left\Vert \delta^{ij}-a^{ij}\right\Vert ^{2}<1.
\]
For higher dimensions, in \cite[Lemma 3]{N54}, Nirenberg stated the following generalization

\begin{theorem}
\label{notproved}Let $U=(u_{1},u_{2},...,u_{n})$ be an $%
\mathbb{R}
^{n}$-valued continuous function defined in a domain $B_{1}\subset$ $%
\mathbb{R}
^{n}$ having continuous first derivatives satisfying
\begin{equation}
\sum_{i,j}u_{i,j}^{2}\leq k\sum_{i,j}\left(  u_{i,j}u_{j,i}-u_{i,i}%
u_{j,j}\right)  +k_{1}%
\end{equation}
and in addition
\[
k<\frac{n-1}{n-2}.
\]
Then the functions $u_{i}$ are H\"{o}lder continuous on the interior domain.
\end{theorem}

The proof of Theorem \ref{notproved} would follow from an integral estimate of
the form Lemma \ref{nlemma}. However a proof is not given, although it is
stated \cite[Section 3]{N54} that the proof of Theorem \ref{notproved} is
\textquotedblleft similar"to the proof of Lemma \ref{nlemma}.

In any case, the result of Cordes in 1956 \cite[page 292]{Cordes} provides
better constants: Cordes defines the $K_{\varepsilon}^{\prime}$-condition for
a symmetric matrix with eigenvalues $\lambda_{1},...\lambda_{n}$ as:%
\[
(n-1)\left(  1+\frac{n(n-2)}{(n+1)(n-1)}\right)  \sum_{i<k}(\lambda
_{i}-\lambda_{k})^{2}\leq(1-\varepsilon)\left(  \sum_{i}\lambda_{i}\right)
^{2}.
\]
Cordes proves the following \cite[Satz 8, page 303]{Cordes} :

\begin{theorem}
Suppose the coefficients $a^{ij}$ satisfy a $K_{\varepsilon}^{\prime}%
$-condition. There exists an $\alpha$ depending on $\varepsilon$ such that the
solutions to (\ref{linf}) satisfy an estimate of the form
\[
\left\Vert u\right\Vert _{C^{1,\alpha}(B_{1/2})}<c\left(  \left\Vert
f\right\Vert _{L^{\infty}(B_{1})}+\left\Vert u\right\Vert _{L^{\infty}(B_{1}%
)}\right)  .
\]

\end{theorem}

The proof involves pages of integrals. In 1961, \cite[Theorem 2]{C61}, Cordes
offered an outline for a refined argument, and summarized the results (in English).\ 

The \textquotedblleft Cordes condition" in the literature is often phrased as
the following:
\begin{equation}
\left\Vert A\right\Vert _{HS}^{2}<\frac{1}{n-1+\delta}\left\vert
Tr(A)\right\vert ^{2}.\label{ccond}%
\end{equation}
Note that this is equivalent (for $\varepsilon$ not equal to but depending on
$\delta$) to the $K_{\varepsilon}$-condition defined by Cordes in \cite[page
292]{Cordes}:%
\[
(n-1)\sum_{i<k}(\lambda_{i}-\lambda_{k})^{2}\leq(1-\varepsilon)\left(
\sum_{i}\lambda_{i}\right)  ^{2}.
\]
Cordes showed solutions to (\ref{linf}) will be $C^{\alpha}$ for $f$
bounded$.$ Talenti \cite{Talenti} applied this condition to show that
solutions to (\ref{linf}) exist in $W^{2,2}$ when $f\in L^{2}.$ \ 

It is interesting to look at the linearized operator for nonlinear equations
of the form (\ref{f}), in particular when equation (\ref{f}) is neither convex
nor concave. If the linearized operator satisfies a $K_{\varepsilon}^{\prime}%
$-condition, then $C^{3}$ solutions will be $C^{2,\alpha}$ with uniform
estimates based on the $C^{1}$ norm.

In general, a regularity boosting with estimates for equations of the form
(\ref{f}) can follow by applying Cordes-Nirenberg type results, locally, to
smooth solutions, even when the operator does not globally satisfy such a
condition. For a given nonlinear equation one may differentiate (\ref{f}).
When the oscillation of the linearized operator $F^{ij}$ depends continuously
on the oscillation of $D^{2}u,$ there will be a $\delta_{0}$ such that if the
oscillation of the Hessian is smaller than $\delta_{0}$ the oscillation of
$F^{ij}$ will be less than $\varepsilon_{0},$ thus $C^{2,\alpha}$ estimates
apply. In particular, any modulus of continuity on the Hessian can be used to
derive H\"{o}lder continuity: Essentially, the results in \cite{CaoLiWang} can
be \textquotedblleft quantized". (Keep in mind that we may alway use a
transformation like the one following (\ref{ataw}), locally, so that the
equation satisfies a $K_{\varepsilon}^{\prime}$-condition nearby).
Bootstrapping, using Schauder theory on difference quotients, one can derive
estimates of all orders. In particular, the full suite of estimates can be
derived by knowing the Hessian is nearly continuous.

\bigskip We record the following corollary which follows immediately from this
discussion$.$

\begin{corollary}
Suppose that $u$ is a entire quadratic solution to $F\left(  D^{2}u\right)
=0,$ for $F\in C^{1,\beta}.$ Then there is an $\varepsilon_{0}\left(
\left\Vert F\right\Vert _{C^{1,\beta}},n\right)  >0$ such that any solution
$u^{\prime}$ with
\[
\left\Vert D^{2}u-D^{2}u^{\prime}\right\Vert <\varepsilon_{0}%
\]
must also be quadratic.
\end{corollary}

Thus quadratic solutions are rigid with respect to the global $C^{2}$ norm.
\newline

\textbf{Acknowledgments.} The first author is grateful to Professor Yu Yuan
for discussions.

\bibliographystyle{amsalpha}
\bibliography{alop}

\providecommand{\bysame}{\leavevmode\hbox to3em{\hrulefill}\thinspace}
\providecommand{\MR}{\relax\ifhmode\unskip\space\fi MR }
\providecommand{\MRhref}[2]{%
  \href{http://www.ams.org/mathscinet-getitem?mr=#1}{#2}
}
\providecommand{\href}[2]{#2}
\begin{thebibliography}{CLW11}

\bibitem[ASS]{SSA}
Scott~N. Armstrong, Luis~E. Silvestre, and Charles~K. Smart, \emph{Partial
  regularity of solutions of fully nonlinear, uniformly elliptic equations},
  Communications on Pure and Applied Mathematics \textbf{65}, no.~8,
  1169--1184.

\bibitem[Cal58]{C}
Eugenio Calabi, \emph{Improper affine hyperspheres of convex type and a
  generalization of a theorem by {K}. {J}\"{o}rgens}, Michigan Math. J.
  \textbf{5} (1958), 105--126. \MR{0106487}

\bibitem[CC95]{CC}
Luis~A Caffarelli and Xavier Cabr{\'e}, \emph{Fully nonlinear elliptic
  equations}, vol.~43, American Mathematical Soc., 1995.

\bibitem[CC03]{CC03}
Xavier Cabr\'{e} and Luis~A. Caffarelli, \emph{Interior {$C^{2,\alpha}$}
  regularity theory for a class of nonconvex fully nonlinear elliptic
  equations}, J. Math. Pures Appl. (9) \textbf{82} (2003), no.~5, 573--612.
  \MR{1995493}

\bibitem[CLW11]{CaoLiWang}
Yi~Cao, DongSheng Li, and LiHe Wang, \emph{A priori estimates for classical
  solutions of fully nonlinear elliptic equations}, Sci. China Math.
  \textbf{54} (2011), no.~3, 457--462. \MR{2775423}

\bibitem[Col16]{Collins}
Tristan~C. Collins, \emph{{$C^{2,\alpha}$} estimates for nonlinear elliptic
  equations of twisted type}, Calc. Var. Partial Differential Equations
  \textbf{55} (2016), no.~1, Art. 6, 11. \MR{3441283}

\bibitem[Cor56]{Cordes}
Heinz~Otto Cordes, \emph{{\"U}ber die erste randwertaufgabe bei quasilinearen
  differentialgleichungen zweiter ordnung in mehr als zwei variablen},
  Mathematische Annalen \textbf{131} (1956), no.~3, 278--312.

\bibitem[Cor61]{C61}
H.~O. Cordes, \emph{Zero order a priori estimates for solutions of elliptic
  differential equations}, Proc. {S}ympos. {P}ure {M}ath., {V}ol. {IV},
  American Mathematical Society, Providence, R.I., 1961, pp.~157--166.
  \MR{0146511}

\bibitem[CY00]{CafYuan}
Luis~A. Caffarelli and Yu~Yuan, \emph{A priori estimates for solutions of fully
  nonlinear equations with convex level set}, Indiana Univ. Math. J.
  \textbf{49} (2000), no.~2, 681--695. \MR{1793687}

\bibitem[Eva82]{E}
Lawrence~C. Evans, \emph{Classical solutions of fully nonlinear, convex,
  second-order elliptic equations}, Communications on Pure and Applied
  Mathematics \textbf{35} (1982), no.~3, 333--363.

\bibitem[GT01]{GT}
David Gilbarg and Neil~S. Trudinger, \emph{Elliptic partial differential
  equations of second order}, Classics in Mathematics, Springer-Verlag, Berlin,
  2001, Reprint of the 1998 edition. \MR{1814364}

\bibitem[KS81]{KS}
N~V Krylov and M~V Safonov, \emph{A certain property of solutions of parabolic
  equations with measurable coefficients}, Mathematics of the USSR-Izvestiya
  \textbf{16} (1981), no.~1, 151--164.

\bibitem[Nir53]{N53}
Louis Nirenberg, \emph{On nonlinear elliptic partial differential equations and
  {H}\"{o}lder continuity}, Comm. Pure Appl. Math. \textbf{6} (1953), 103--156;
  addendum, 395. \MR{0064986}

\bibitem[Nir54]{N54}
L.~Nirenberg, \emph{On a generalization of quasi-conformal mappings and its
  application to elliptic partial differential equations}, Contributions to the
  theory of partial differential equations, Annals of Mathematics Studies, no.
  33, Princeton University Press, Princeton, N. J., 1954, pp.~95--100.
  \MR{0066532}

\bibitem[N.V83]{K}
N.V.Krylov, \emph{Boundedly nonhomogeneous elliptic and parabolic equations},
  Math. USSR-Izv. \textbf{20} (1983), no.~3, 459--492.

\bibitem[NV08]{NV}
Nikolai Nadirashvili and Serge Vl\u{a}du\c{t}, \emph{Singular viscosity
  solutions to fully nonlinear elliptic equations}, J. Math. Pures Appl. (9)
  \textbf{89} (2008), no.~2, 107--113. \MR{2391642}

\bibitem[NV10]{NVSL}
Nikolai Nadirashvili and Serge Vlăduţ, \emph{Singular solution to special
  lagrangian equations}, Annales de l'Institut Henri Poincare (C) Non Linear
  Analysis \textbf{27} (2010), no.~5, 1179 -- 1188.

\bibitem[Pin16]{Pingali}
Vamsi~P. Pingali, \emph{{$C^{2,\alpha}$} estimates and existence results for a
  nonconcave {PDE}}, Electron. J. Differential Equations (2016), Paper No. 168,
  10. \MR{3522223}

\bibitem[Sav07]{SV}
Ovidiu Savin, \emph{Small perturbation solutions for elliptic equations},
  Communications in Partial Differential Equations \textbf{32} (2007), no.~4,
  557--578.

\bibitem[SW16]{StreetsWarren}
Jeffrey Streets and Micah Warren, \emph{Evans-{K}rylov estimates for a
  nonconvex {M}onge-{A}mp\`ere equation}, Math. Ann. \textbf{365} (2016),
  no.~1-2, 805--834. \MR{3498927}

\bibitem[Tal65]{Talenti}
Giorgio Talenti, \emph{Sopra una classe di equazioni ellittiche a coefficienti
  misurabili}, Ann. Mat. Pura Appl. (4) \textbf{69} (1965), 285--304.
  \MR{0201816}

\bibitem[Yua01]{Yuan2001}
Yu~Yuan, \emph{A priori estimates for solutions of fully nonlinear special
  {L}agrangian equations}, Ann. Inst. H. Poincar\'{e} Anal. Non Lin\'{e}aire
  \textbf{18} (2001), no.~2, 261--270. \MR{1808031}

\end{thebibliography}

\end{document}